\documentclass[reqno,11pt]{amsart}
\usepackage{latexsym}
\usepackage{fullpage}
\usepackage{amssymb}
\usepackage{amscd}
\usepackage{hyperref}
\def\R{\ensuremath{\mathbb R}}
\def\C{\ensuremath{\mathbb C}}

\def\Z{\ensuremath{\mathbb Z}}

\def\rk{\operatorname{rk}}
\def\ind{\operatorname{ind}}

\def\rk{\operatorname{rk}}
\def\Ad{\operatorname{Ad}}
\def\ch{\operatorname{ch}}
\def\frg{{\mathfrak g}}
\def\frh{{\mathfrak h}}

\def\frp{{\mathfrak p}}
\def\frs{{\mathfrak s}}
\def\frt{{\mathfrak t}}
\def\frz{{\mathfrak z}}

\def\norm#1{\left\|#1\right\|}

\def\<{\langle}
\def\>{\rangle}

\def\End{{\operatorname{End}}}
\newcommand\vol{\operatorname{vol}}

\newcommand\im{\operatorname{im}}

\def\SO{\mathrm{SO}}
\def\Spin{\mathrm{Spin}}

\def\spin{\mathfrak{spin}}

\def\Adach{{\hat A}}
\def\Ad{\operatorname{Ad}}
\def\ad{\operatorname{ad}}
\def\ads{{\,\widetilde{\!\vphantom{x}\smash{\mathrm{ad}}\!}\,}}
\def\adsh{{\,\widehat{\widetilde{\!\vphantom{x}\smash{\mathrm{ad}}\!}}\,}}

\def\adx_#1{\ads_{\frp,#1}}
\def\ady_#1{\adsh{}_{\frp,#1}}

\def\punkt{{\mathord{\,\cdot\,}}}
\def\even{{\mathrm{ev}}}
\def\odd{{\mathrm{odd}}}
\def\^#1{{\,}^{#1\!}}

\numberwithin{equation}{section}
\swapnumbers
\theoremstyle{plain}
\newtheorem{Lemma}[equation]{Lemma}
\newtheorem{Proposition}[equation]{Proposition}
\newtheorem{Corollary}[equation]{Corollary}

\newtheorem{thm}{Theorem}
\newtheorem{cor}[thm]{Corollary}

\theoremstyle{definition}
\newtheorem{Definition}[equation]{Definition}
\newtheorem*{Dfn}{Definition}

\theoremstyle{remark}
\newtheorem{rem}{Remark}
\newtheorem{Remark}[equation]{Remark}
\newtheorem{Example}[equation]{Example}

\begin{document}
\title{Scalar Curvature Estimates\\by Parallel Alternating Torsion}
\author{Sebastian Goette}
\address{Mathematisches Institut\\
Universit\"at Freiburg\\
Eckerstr.~1\\
79104 Freiburg\\
Germany}
\email{sebastian.goette@math.uni-freiburg.de}
\thanks{Supported in part by DFG special programme
``Global Differential Geometry''}
\begin{abstract}
We generalize Llarull's scalar curvature comparison to Riemannian
manifolds admitting metric connections with parallel and alternating torsion
and having a nonnegative curvature operator on~$\Lambda^2TM$.
As a byproduct,
we show that Euler number and signature of such manifolds
are determined by their global holonomy representation.
Our result holds in particular for all quotients of compact Lie groups of
equal rank, equipped with a normal homogeneous metric.

We also correct a mistake in the treatment of odd-dimensional spaces
in~\cite{Gvawi} and~\cite{GS}.
\end{abstract}

\maketitle
There is a well known relation between the existence of
metrics of positive scalar curvature on a closed manifold~$M$
and the topology of~$M$.
If there exist metrics with positive scalar curvature~$\kappa$,
one would like to measure how large~$\kappa$ can become.
This could be done using the Yamabe number or $\sigma$-invariant~$\sigma(M)$,
which is defined by taking the infimum of the integral of~$\kappa$
for all metrics in a conformal class,
and then the supremum over all conformal classes.
As announced by Ammann, Dahl and Humbert in~\cite{ADH},
there exists a constant~$c_m>0$ for~$m=\dim M$
such that~$\min(\sigma(M),c_m)$ is a spin bordism invariant
over~$B\pi_1(M)$.
In this note however,
we will consider the pointwise scalar curvature instead.

Let~$g$ and~$\bar g$ be two Riemannian metrics on~$M$,
and let~$\kappa$ and~$\bar\kappa$ their scalar curvature.
Using the $K$-area inequalities, Gromov showed 
in~\cite{gromov} that there is a finite upper bound for~$\min\bar\kappa$
if~$\bar g\ge g$ on~$\Lambda^2TM$.
A first example for a sharp upper bound was given
by Llarull in \cite{Llarull}.
If~$g$ denotes the round metric on~$S^n$ and~$\bar g\ge g$ on~$\Lambda^2TM$,
then~$\bar\kappa(p)\le\kappa(p)$ for some~$p\in M$.
In fact, $\bar\kappa\ge\kappa$ on~$M$ implies~$\bar\kappa=\kappa$
and~$\bar g=g$.
In other words,
the round metric on the sphere is strongly area-extremal
in the following sense.

\begin{Dfn}
  A metric~$g$ on~$M$ is called {\em area-extremal\/}
  if for all metrics~$\bar g$ on~$M$ with~$\bar g\ge g$ on~$\Lambda^2TM$,
  the inequality~$\bar\kappa\ge\kappa$ everywhere on~$M$
  implies that~$\bar\kappa=\kappa$.
  We call~$g$ {\em strongly area-extremal\/} if~$\bar g\ge g$ on~$\Lambda^2TM$
  and~$\bar\kappa\ge\kappa$ on~$M$ also imply that~$\bar g=g$.
\end{Dfn}

In \cite{gromov}, Gromov asked which manifolds possess area-extremal 
metrics and how such metrics may look like. He conjectured that Riemannian 
symmetric spaces should have area-extremal metrics.
He also proposed to investigate not only variations of the metric
on~$M$ itself,
but to consider also area-non-increasing spin maps of non-vanishing
$\Adach$-degree from other Riemannian manifolds to~$M$,
see Section~\ref{MapSection} for an explanation of these terms.

\begin{Dfn}
  A metric~$g$ on~$M$ is called {\em area-extremal in the sense of Gromov\/}
  if for all smooth spin maps~$f\colon(N,\bar g)\to(M,g)$
  of nonzero $\hat A$-degree with~$\bar g\ge f^*g$ on~$\Lambda^2TN$,
  ine inequality~$\bar\kappa\ge\kappa\circ f$ everywhere on~$N$
  implies that~$\bar\kappa=\kappa\circ f$.
  We call~$g$ {\em strongly area-extremal in the sense of Gromov\/}
  if~$\bar g\ge f^*g$ on~$\Lambda^2TN$
  and~$\bar\kappa\ge\kappa\circ f$ on~$N$ also imply that~$f$
  is a Riemannian submersion.
\end{Dfn}

One might think of stronger versions of area-extremality.
However, already if~$M$ is a point, one has to assume
that~$f$ is spin and has nonzero degree in one sense or another.
Llarull proved in~\cite{Llarull2} that spheres are also area-extremal
in the sense of Gromov.
Since then,
several other manifolds have been shown to be area-extremal,
some of them even in the sense of Gromov,
see~\cite{GS} for an overview of related results.

So far, the only method to prove area-extremality
combines the nonvanishing of an analytic index
with a careful investigation of the curvature term in
the Bochner-Lichnerowicz-Weitzenb\"ock formula
for a twisted Dirac operator~$\bar D$ with respect to the metric~$\bar g$.
The relevant indices will be discussed in Section~\ref{IndexSect}.
As observed in~\cite{GS2}, \cite{GS},
the curvature conditions needed for the scalar curvature estimates
are satisfied in the following cases.
\begin{enumerate}
\item $(M,g)$ is a Fano manifold,
  i.e., it is K\"ahler and has nonnegative Ricci curvature \cite{GS2}.
\item The curvature operator on~$\Lambda^2TM$ is nonnegative, see~\cite{GS}.
  Then~$M$ is homeomorphic to a Riemannian symmetric space~$G/H$.
\end{enumerate}
In the following situation,
one expects the estimate to hold as well.
\begin{enumerate}
\item[(3)] $M$ is quaternionic K\"ahler with positive scalar curvature.
\end{enumerate}
However,
it is conjectured that quaternionic K\"ahler manifolds
with positive scalar curvature are symmetric with~$\rk G=\rk H$.
If this is true, then~(3) is just a special case of~(1).
The estimate of~\cite{GS} has recently been strengthened as follows.

\begin{thm}[Listing, \cite{Listing}]
  Let~$(M,g)$ be a compact, connected, oriented Riemannian manifold
  with non-negative curvature operator 
  on~$\Lambda^2(TM)$.
  Assume that there exists a parallel Dirac subbundle~$W$
  of~$\Lambda^\bullet T^*M$
  such that the Hodge-Dirac operator has non-vanishing index on~$W$.
  Let~$\bar g$ be another Riemannian metric on~$M$
  with scalar curvature~$\bar\kappa$,
  and let~$f$ be a positive function on~$M$
  such that~$f\,\bar g\ge g$ on~$\Lambda^2TM$.
  Then~$\bar\kappa\ge f\,\kappa$ implies~$\bar\kappa=f\,\kappa$.
  If moreover, the Ricci curvature of~$g$ satisfies~$\rho>0$
  and~$2\rho-\kappa<0$,
  then~$\bar\kappa\ge f\,\kappa$ implies~$f\,\bar g=g$
  and~$f$ is constant.
\end{thm}

In the present paper, we replace the Levi-Civita connection
of the metric~$\bar g$
in the definition of the twisted Dirac operator~$\bar D$ on~$M$
(or~$N$ if we are considering maps~$f\colon N\to M$)
by a different connection~$\tilde\nabla$.
This is still a generalised Dirac operator in the sense that its
square is the sum of a generalised Laplacian and a zero order operator.
The generalised Laplacian here is induced by a connection~$\nabla'$ on~$TM$.
We assume that~$\nabla'$ has parallel and alternating torsion tensor~$T$,
as explained in Section~\ref{PartorSect}.
Unless~$T=0$,
this implies that the holonomy group of~$\nabla'$ is a proper subgroup
of~$SO(TM)$.

Because~$T$ is parallel,
the curvature~$R'$ of~$\nabla'$
acts as a symmetric endomorphism on~$\Lambda^2TM$
by Lemma~\ref{ParTorLemma}~\eqref{SymmetryRel}.
Hence, there exists a ``curvature operator'' as in the Riemannian case.
For our estimates, this operator must be nonnegative.

Let us also define the kernel of the torsion tensor~$T$ as
\begin{equation*}
  \ker T_p:=\bigl\{\,V\in T_pM\bigm|T(V,W)=0\text{ for all }W\in T_pM\,\bigr\}
\end{equation*}
for all~$p\in M$.
Note that~$\ker T_p$ defines a $\nabla'$-parallel subbundle of~$TM$.
In particular,
the universal covering of the triple~$(M,g,\nabla')$ always splits
as a product~$(M_0,g_0,\nabla'_0)\times(M_1,g_1,\nabla'_1)$
such that~$\nabla'_0$ is the Levi-Civita connection
and~$\nabla'_1$ has parallel and alternating torsion~$T_1$
with~$\ker T_1=0$.
This splitting will not be used in this paper.

\begin{thm}\label{MainThm}
  Let~$(M,g)$ be a closed connected Riemannian spin manifold
  with scalar curvature~$\kappa$.
  Assume that there exists a metric connection~$\nabla'$ on~$TM$
  with parallel and alternating torsion that induces
  a nonnegative curvature operator~$R'$ on~$\Lambda^2TM$.
  Assume also
  that there exists a $\nabla'$-parallel
  Dirac subbundle~$W\subset\Lambda^\bullet T^*M$
  such that the restriction of the Hodge Dirac to~$W$ has nonzero index.
  Then~$(M,g)$ is area-extremal.

  Moreover, if either
  \begin{enumerate}
  \item\label{NewCond} the Ricci tensor~$\rho$ of~$g$ is positive definite
    on~$\ker T$ and~$T\ne 0$, or
  \item\label{GSCond} we have~$\rho>0$ and~$2\rho-\kappa\,g<0$,
  \end{enumerate}
  then~$(M,g)$ is strongly area-extremal.
\end{thm}

The index above will be explained in Section~\ref{IndexSect},
and the theorem will be proved in Section~\ref{ProofSect}.

As mentioned above,
our proof relies on a generalisation of the Bochner-Lichnerowicz-Weitzenb\"ock
formula for a certain modified twisted Dirac operator~$\tilde D$.
For connections with parallel and alternating torsion,
this formula is stated in Corollary~\ref{BLWCor}.
Let us state a simple consequence of this formula
that we prove in Section~\ref{BLWSection}.

\begin{cor}\label{HolonomyCor}
  Let~$(M,g)$ be a connected closed Riemannian manifold
  admitting a connection~$\nabla'$ with parallel and alternating torsion
  and nonnegative curvature operator on~$\Lambda^2TM$.
  Then~$\tilde D$ preserves the subspace of~$\Omega^\bullet(M)$
  consisting of $\nabla'$-parallel forms,
  and all $\tilde D$-harmonic forms are $\nabla'$-parallel.
  In particular,
  the Euler characteristic, the signature and the Kervaire semicharacteristic
  of~$M$---whenever they are defined---can be read off
  from the global holonomy representation of~$\nabla'$.
\end{cor}

Under slightly stronger assumptions than in Theorem~\ref{MainThm},
we can even prove that~$(M,g)$ is area-extremal in the sense of Gromov.

\begin{thm}\label{MapThm}
  Let~$(M,g)$ be as in Theorem~\ref{MainThm}.
  If~$M$ is oriented and has positive Euler number,
  then~$(M,g)$ is area-extremal in the sense of Gromov.

  Moreover, if either
  \begin{enumerate}
  \item\label{NewCond2} the Ricci tensor~$\rho$ of~$g$ is positive definite
    on~$\ker T$ and~$T\ne 0$, or
  \item\label{GSCond2} we have~$\rho>0$ and~$2\rho-\kappa\,g<0$,
  \end{enumerate}
  then~$(M,g)$ is strongly area-extremal in the sense of Gromov.
\end{thm}

This will be proved in Section~\ref{MapSection}.
In Section~\ref{HomoSect},
we apply Theorems~\ref{MainThm} and~\ref{MapThm}
to quotients of compact Lie groups.

\begin{cor}\label{HomoThm}
  Let~$M=G/H$ be a quotient of compact connected Lie groups,
  and let~$g$ be a normal homogeneous metric.
  If~$\rk G=\rk H$
  then~$(M,g)$ is area-extremal in the sense of Gromov,
  and strongly area-extremal in the sense of Gromov if moreover~$\dim M>2$.

  If~$\rk G-\rk H=1$, $\dim M=4k+1$ and~$W\subset\Lambda^\bullet T^*M$
  is a homogeneous Dirac subbundle such that~$D|_W$ has nonvanishing index,
  then~$(M,g)$ is area-extremal.
  If moreover~$(M,g)$ does not contain a Euclidean local de Rham factor
  and~$\dim M>2$,
  then~$(M,g)$ is strongly area-extremal.
\end{cor}

Note that the proof of the main results in~\cite{GS} and~\cite{Gvawi}
for symmetric spaces with~$\rk G-\rk H=1$ only work
with the same index-theoretic condition as in the Corollary above,
see Section~\ref{IndexSect} for an explanation.


Let us now relate the various assumptions on~$\nabla'$ in Theorem~\ref{MainThm}
to more geometric properties of~$(M,g)$.

\begin{rem}\label{ParTorRem}
  If~$T=0$,
  then~$\nabla'$ is the Levi-Civita connection,
  and we are in the situation of~\cite{GS}.
  Hence we assume that~$T\ne 0$ is $\nabla'$-parallel and alternating.
  In particular,
  it is fixed by the holonomy group~$H$ of~$\nabla'$.
  Let~$\pi$ denote the holonomy representation of~$H$. 
%
%
  If we assume that~$\pi$ is irreducible,
  then by a result of Cleyton and Swann~\cite{CS},
  there are three possibilities.
  \begin{enumerate}
  \item $(M,g)$ is locally isometric to a non-symmetric,
    isotropy irreducible homogeneous space~$G/H$;
  \item $(M,g)$ is locally isometric to one of the irreducible
    symmetric spaces~$H\times H/H$ or~$H^{\mathbb C}/H$;
  \item $(M,g)$ has weak holonomy~$SU(3)$ or~$G_2$.
  \end{enumerate}
  We are not aware of a general classification result in the case
  where~$H$ does not act irreducibly.
  If we assume that~$\dim M=2n$ and~$H\subset U(n)$,
  then~$T\ne 0$ is alternating and $\nabla'$-parallel
  iff~$M$ is nearly K\"ahler by a result of Kirichenko~\cite{Kiri};
  these manifolds have been further classified by Nagy~\cite{Nagy}
  and others.
%
\end{rem}

\begin{rem}\label{PosRem}
  If~$\nabla'$ has parallel and alternating torsion,
  we have a well-defined symmetric
  ``curvature operator''~$R'\in\End(\Lambda^2TM)$.
  We assume that~$R'$ is nonnegative.
  In the case where~$T=0$ and~$M$ is locally irreducible,
  one knows that~$M$ is either homeomorphic to a sphere,
  biholomorphic to a complex projective space, or isometric
  to a Riemannian symmetric space of compact type
  by results of Gallot and Meyer~\cite{gal}, Cao and Chow~\cite{CC}
  and Tachibana~\cite{T}.
  If~$T=0$ and the curvature operator is strictly positive,
  then~$M$ is actually diffeomorphic to a sphere
  by a recent result of B\"ohm and Wilking~\cite{BW}.

  We will see in section~\ref{HomoSect} that on all normal homogeneous spaces,
  the reductive connection has parallel and alternating torsion
  and its curvature operator is nonnegative.
  Most normal homogeneous metrics do not have nonnegative
  Riemannian curvature operator,
  so we really obtain some new examples of area-extremal metrics.
  But by Lemma~\ref{ParTorLemma}~\eqref{SectRel} below,
  nonnegativity of~$R'$ still implies
  that the Riemannian sectional curvatures of~$(M,g)$ are nonnegative.
  In particular,
  we are still far away from the goal stated in~\cite{GS}
  to generalize Llarull's theorem to Ricci positive manifolds.
%
\end{rem}

\begin{rem}\label{RigidityRem}
  Let us now regard our sufficient conditions for strong area-extremality
  in Theorems~\ref{MainThm} and~\ref{MapThm}.
  Obviously, condition~\eqref{NewCond} follows from~\eqref{GSCond} if~$T\ne0$.
  We will see that condition~\eqref{NewCond} also
  implies condition~\eqref{GSCond},
  which was used in~\cite{GS} in the case~$T=0$.

  Assume that~$U\notin\ker T$.
  Using Lemma~\ref{ParTorLemma}~\eqref{SectRel}, we see that
  \begin{equation*}
    \rho(U,U)=\sum_{i=1}^m\<R_{U,e_i}e_i,U\>
    =\sum_{i=1}^m\bigl(\<R'_{U,e_i}e_i,U\>+\|T(U,e_i)\|^2\bigr)
    >0
  \end{equation*}
  because~$\<R'_{U,e_i}e_i,U\>\ge 0$ and~$\|T(U,e_i)\|^2>0$
  for at least one~$i$.
  So~$(M,g)$ is indeed Ricci positive.

  Let~$U$ now be an arbitrary unit vector
  and extend~$U=e_1$ to a $g$-orthonormal frame.
  Then
  \begin{equation*}
    (\kappa\,g-2\rho)(U,U)
    =\sum_{i,j=2}^mR_{ijji}
    =\sum_{i,j=2}^m\biggl(R'_{ijji}+\frac14\sum_{k=1}^m\tau_{ijk}^2\biggr)>0
  \end{equation*}
  because~$R'_{ijji}\ge0$ by assumption,
  and at least one~$\tau_{ijk}$ does not vanish.
  Thus~$2\rho-\kappa$ is indeed negative definite.
\end{rem}

\subsection*{Acknowledgements}
The author wishes to thank Bernd Ammann, Mario Listing and Uwe Semmelmann
for helpful conversations and interest in this project.

\section{Metric Connections with Parallel and Alternating Torsion}
\label{PartorSect}
We collect some elementary properties of connections with parallel
alternating torsion.
Let~$\nabla=\nabla^{TM}$ denote the Levi-Civita connection on~$M$.

\begin{Definition}\label{ParTorDef}
  Let~$(M,g)$ be a Riemannian manifold,
  and let~$\nabla^{\prime TM}$ be a metric connection on~$TM$
  with torsion tensor~$T$.
  If~$\<T(\punkt,\punkt),\punkt\>$ is alternating,
  we say that~$\nabla'$ has {\em alternating torsion.\/}
  If~$\nabla'T=0$, we say that~$\nabla'$ has {\em parallel torsion.\/}
\end{Definition}

We will omit the superscript~$TM$ whenever no ambiguity can arise.
Let~$X$, $Y$, $Z$, $W$ be smooth vector fields on~$M$.
Let us define~$\tau\in\Omega^3(M)$ by
\begin{equation*}
  \tau(X,Y,Z)=\<T(X,Y),Z\>\;.
\end{equation*}
The curvature~$R'$ of~$\nabla'$ acts on~$\Lambda^2TM$ such that
\begin{equation*}
  \<R'(X\wedge Y),Z\wedge W\>=-\<R_{X,Y}Z,W\>\;.
\end{equation*}
For the Levi-Civita connection,
this gives the Riemannian curvature operator on~$\Lambda^2TM$.

\begin{Lemma}\label{ParTorLemma}
  Assume that~$\nabla'$ has parallel and alternating torsion.
  Then
  \begin{enumerate}
  \item\label{NablaAlphaRel}
    The form~$\nabla\tau=\frac14\,d\tau$ is fully alternating.
  \item\label{RiccatiRel}
    The form~$d\tau$ is given by
    \begin{equation*}
      (d\tau)(X,Y,Z,W)
      =2\bigl(\<T(X,Y),T(Z,W)\>+\<T(Y,Z),T(X,W)\>+\<T(Z,X),T(Y,W)\>\bigr)\;,
    \end{equation*}
  \item\label{CurvFormel}
    The curvature of~$R'$ is given by
    \begin{equation*}
      R'_{X,Y}Z
      =R_{X,Y}Z+(\nabla_XT)(Y,Z)
        +\frac14\,T\bigl(X,T(Y,Z)\bigr)-\frac14\,T\bigl(Y,T(X,Z)\bigr)\;.
    \end{equation*}
  \item\label{SymmetryRel}
    The action of~$R'$ on~$\Lambda^2TM$ is symmetric.
  \item\label{SectRel}
    We have the relation
    \begin{equation*}
      \<R'_{X,Y}Y,X\>
      =\<R_{X,Y}Y,X\>-\frac14\norm{T(X,Y)}^2\;.
    \end{equation*}
  \item\label{BianchiRel}
    The tensor~$S=R'-T(T(\punkt,\punkt),\punkt)$ has the same symmetries
    as the Riemannian curvature tensor.
  \end{enumerate}
\end{Lemma}

Assertion~\eqref{BianchiRel} will not be needed later on.

\begin{proof}
  If~$\nabla'$ has alternating torsion, then
  \begin{equation}\label{NablaFormel}
    \nabla'_XY=\nabla_XY+\frac12\,T(X,Y)\;.
  \end{equation}
  By a routine computation,
  the curvature tensor~$R'$ of~$\nabla'$ is given as
  \begin{multline}\label{CurvatureFormel}
    R'_{X,Y}Z
    =R_{X,Y}Z+\frac12\,(\nabla_XT)(Y,Z)-\frac12\,(\nabla_YT)(X,Z)\\
      +\frac14\,T\bigl(X,T(Y,Z)\bigr)-\frac14\,T\bigl(Y,T(X,Z)\bigr)\;.
  \end{multline}

  If~$\nabla'$ has parallel torsion,
  then
  \begin{equation}\label{NablaTFormel}
    0=(\nabla_XT)(Y,Z)+\frac12\,\Bigl(T\bigl(X,T(Y,Z)\bigr)
      +T\bigl(Y,T(Z,X)\bigr)+T\bigl(Z,T(X,Y)\bigr)\Bigr)\;,
  \end{equation}
  which shows that~$\nabla T$ is fully alternating.
  Because~$\nabla$ is metric and~$T$ is alternating,
  the form~$\nabla\tau=\<(\nabla_\punkt T)(\punkt,\punkt),\punkt\>$
  is also fully alternating.
  Thus
  \begin{multline*}
    (d\tau)(X,Y,Z,W)
    =(\nabla_X\tau)(Y,Z,W)-(\nabla_Y\tau)(X,Z,W)\\
      +(\nabla_Z\tau)(X,Y,W)-(\nabla_W\tau)(X,Y,Z)
    =4(\nabla_X\tau)(Y,Z,W)\;,
  \end{multline*}
  which gives~\eqref{NablaAlphaRel}.
  Inserting this into~\eqref{NablaTFormel} gives
  \begin{align*}
    (d\tau)(X,Y,Z,W)
    &=4\,\<(\nabla_XT)(Y,Z),W\>\\
    &=2\bigl(\<T(X,Y),T(Z,W)\>+\<T(Y,Z),T(X,W)\>+\<T(Z,X),T(Y,W)\>\bigr)\;,
  \end{align*}
  which is~\eqref{RiccatiRel}.

  From~\eqref{CurvatureFormel}, we get~\eqref{CurvFormel} by
  \begin{multline*}
    \<R'_{X,Y}Z,W\>
    =\<R_{X,Y}Z,W\>+\<(\nabla_XT)(Y,Z),W\>\\
      +\frac14\,\<T(X,Z),T(Y,W)\>-\frac14\,\<T(Y,Z),T(X,W)\>
    =\<R'_{Z,W}X,Y\>\;,
  \end{multline*}
  which implies~\eqref{SymmetryRel}
  because~$\nabla\tau$ is fully alternating.
  Inserting~\eqref{NablaTFormel} gives
  \begin{equation}\label{CurvatureFormel2}
    R'_{X,Y}Z
    =R_{X,Y}Z-\frac12\,T\bigl(Z,T(X,Y)\bigr)\\
      -\frac14\,T\bigl(X,T(Y,Z)\bigr)-\frac14\,T\bigl(Y,T(Z,X)\bigr)\;,
  \end{equation}
  from which~\eqref{SectRel} follows.
  Summing over all cyclic permutations proves the Bianchi identity
  in~\eqref{BianchiRel}.
  The other symmetries of~$S$ are obvious.
\end{proof}

\begin{Remark}\label{BianchiRem}
  All assertions we need to prove Theorem~\ref{MainThm} follow
  if~$\nabla'$ has alternating torsion
  and the conclusions~\eqref{NablaAlphaRel} and~\eqref{SymmetryRel}
  of Lemma~\ref{ParTorLemma} hold.
  However,
  from the computations in the proof above it is easy to see that
  these assumptions already imply that~$\nabla'$ has parallel torsion.
\end{Remark}

\section{Index-Theoretical Preliminaries}\label{IndexSect}
We explain the index-theoretical condition used in~\cite{GS2}
and in Theorem~\ref{MainThm}.
We also correct a mistake in the treatment of odd-dimensional spaces
in~\cite{GS} and~\cite{Gvawi}.

Recall that~$\Lambda^\bullet T^*M$ is a Dirac bundle,
with Clifford multiplication given by~$c(X)=\<X,\punkt\>\wedge-\iota_X$
and with a connection~$\nabla$ induced by the Levi-Civita connection~$\nabla$.
By a {\em generalised Dirac subbundle,\/}
we mean a subbundle~$W$ that is closed under this Clifford multiplication.
Note that a Dirac subbundle is also required to be $\nabla$-parallel.
In contrast,
we allow to choose a different Clifford connection~$\tilde\nabla$
on~$\Lambda^\bullet T^*M$
such that~$W$ becomes $\tilde\nabla$-parallel.
Here, a Clifford connection is a connection~$\tilde\nabla$ such that
\begin{equation*}
  \tilde\nabla_X\bigl(c(Y)\alpha\bigr)
  =c\bigl(\nabla_XY\bigr)\alpha+c(Y)\tilde\nabla_X\alpha
\end{equation*}
for all vector fields~$X$, $Y$ and all forms~$\alpha$.
In the following,
we will call the operator~$\tilde D=c\circ\tilde\nabla$
a {\em generalised Hodge-Dirac operator\/} on~$W$.

\begin{Definition}\label{IndexDef}
  We say that a generalised Hodge-Dirac operator~$\tilde D$
  on a generalised Dirac subbundle bundle~$W\subset\Lambda^\bullet T^*M$
  has {\em nonvanishing index,\/}
  if either
  \begin{enumerate}
  \item\label{EvenIndex}
    $M$ is even-dimensional, $W$ splits as~$W^+\oplus W^-$
    such that Clifford multiplication with odd elements
    exchanges~$W^+$ and~$W^-$,
    the operator~$\tilde D$ splits
    as~$\tilde D^\pm\colon\Gamma(W^\pm)\to\Gamma(W^\mp)$,
    and~$0\ne\ind(\tilde D^+)\in\Z$, or
  \item\label{OddIndex}
    $M$ is $4k+1$-dimensional, oriented,
    and~$\dim_\R(\ker\tilde D|_{W\cap\Lambda^\even T^*M})$ is odd.
  \end{enumerate}
\end{Definition}

In situation~\eqref{EvenIndex},
the analytic index is invariant under perturbations
and thus independent of our choice of~$\tilde\nabla$.
In situation~\eqref{OddIndex},
let~$e_1$, \dots, $e_m$ denote a local oriented orthonormal frame,
and let~$\omega_\R=c(e_1)\cdots c(e_m)$ denote
the real Clifford volume element.
Then the operator~$\omega_\R\tilde D$ acts on~$W\cap\Lambda^\even T^*M$
as a real, skew-adjoint Fredholm operator.
In particular, the parity of
\begin{equation*}
  \dim\ker(\tilde D|_{W\cap\Lambda^\even T^*M})
  =\dim\ker(\omega_\R\tilde D|_{W\cap\Lambda^\even T^*M})
\end{equation*}
is again preserved under deformations.

\begin{Example}\label{IndexExample}
  There are three typical situations where the full Hodge-Dirac operator
  on~$W=\Lambda^\bullet T^*M$ has nonvanishing index.
  \begin{enumerate}
  \item\label{EulerCond}
    If~$M$ is even-dimensional and the Euler number~$\chi(M)$ is nonzero,
    we may take~$W^\pm=\Lambda^{\even/\odd}T^*M$.
    This happens for instance if~$G/H$ is a quotient of compact Lie
    groups of equal rank.
  \item\label{SignCond}
    If~$M$ is oriented, $4k$-dimensional, and its signature~$\sigma(M)$
    is nonzero,
    we let~$W^\pm$ be the selfadjoint
    and anti-selfadjoint forms.
    In both cases, we are in situation~\eqref{EvenIndex}.
  \item\label{KervaireCond}
    On the other hand,
    if~$M$ is $(4k+1)$-dimensional and oriented and the Kervaire
    semicharacteristic~$k(M)$ does not vanish,
    then we are in situation~\eqref{OddIndex}.
    The sphere~$S^{4k+1}$ is an example.
  \end{enumerate}
\end{Example}

At this point,
we want to point out a mistake in~\cite{GS} and~\cite{Gvawi}.
In those papers,
we suggested to ``stabilise'' manifolds~$M$ of odd dimension by~$S^{2n}$
to obtain a manifold of dimension~$8k+1$.
Unfortunately,
this does not help if~$\dim M\equiv-1$ mod~$4$.

\begin{Example}\label{OddIndexExample}
  There are situations where a suitable Dirac subbundle
  with a Hodge-Dirac operator of nonvanishing index does not exist.
  Assume that~$M$ and~$N$ are oriented Riemannian manifolds
  of dimensions~$m\equiv-1$ and~$n\equiv 2$ mod~$4$ respectively,
  then~$\dim(M\times N)=4k+1$.
  Let~$\tilde\nabla$ be the tensor product of Clifford connections
  on~$\Lambda^\bullet T^*M$ and~$\Lambda^\bullet T^*N$,
  and let~$W\subset\Lambda^\bullet T^*(M\times N)$ be a $\tilde\nabla$-parallel
  Dirac subbundle.
  Then the deformed Dirac operators~$\tilde D_M$ and~$\tilde D_N$
  both act on~$W$ and anticommute.
  Thus, we have
  \begin{equation*}
    \tilde D^2_{M\times N}=\tilde D^2_M+\tilde D^2_N\qquad\text{and}\qquad
    \ker(\tilde D^2_{M\times N})=\ker(\tilde D^2_M)\cap\ker(\tilde D^2_N)\;.
  \end{equation*}
  The real Clifford volume element~$\omega_{N,\R}$ of~$N$ acts
  on~$\Lambda^\bullet T^*(M\times N)$.
  It is even,
  $\tilde\nabla$-parallel and has square~$\omega_{N,\R}^2=-1$.
  It commutes with~$\tilde D_M$ and anticommutes with~$\tilde D_N$,
  so it acts both on~$\ker(\tilde D^2_M)$ and on~$\ker(\tilde D^2_N)$.
  We may regard~$\omega_{N,\R}$ as a complex structure on the real vector
  space~$\ker(\tilde D^2_{M\times N}|_{W\cap\Lambda^\even T^*M})$,
  which is therefore even-dimensional.
  In particular,
  there are symmetric spaces~$G/H$ with~$\rk G-\rk H=1$
  to which the arguments of~\cite{GS} and~\cite{Gvawi} do not apply.
\end{Example}

\section{Modified Dirac operators}\label{DiracSection}
We construct a modified Dirac operator on Dirac bundles~$(W,\nabla^W)$
over Riemannian manifolds, in particular on the bundle of exterior forms,
starting from a connection~$\nabla'$ with parallel and alternating torsion.
In the special case of homogeneous spaces,
we recover the reductive Dirac operator of~\cite{G1} and~\cite{G2},
also known as Kostant's cubic Dirac operator.
A similar modification of the untwisted Dirac operator has been considered
by Agricola and Friedrich, see~\cite{AF}.

We assume that~$M$ is spin with a fixed spin structure,
and let~$SM\to M$ denote its complex spinor bundle.
Then~$(W,\nabla^W)$
is the tensor product of the spinor bundle~$(SM,\nabla^{SM})$
and another Hermitian vector bundle~$(V,\nabla^V)$.
If~$W$ is a Dirac subbundle of~$\Lambda^\bullet T^*M$,
then we may regard~$V$ as a subbundle of~$SM$ because
\begin{equation}\label{LambdaFormel}
  \Lambda^\bullet T^*M\otimes_\R\C=
  \begin{cases}
    SM\otimes_\C SM	&\text{if $\dim M$ is even, and}\\
    2(SM\otimes_\C SM)	&\text{if $\dim M$ is odd.}
  \end{cases}
\end{equation}
If~$M$ is not spin,
then still all Dirac bundles locally look like twisted spinor bundles.
Since all our constructions in the next chapters are local,
there is no loss of generality in assuming~$M$ to be spin.

Let~$\nabla^{TM}$ denote the Levi-Civita connection on~$M$.
Let us assume that we have fixed~$q\in P_{\Spin}$ over~$p\in M$,
and let us assume that we have a local section~$s$ of~$P_{\Spin}$
such that~$\im d_ps\subset\ker\omega$.
Let us choose a coordinate system near~$p$ such that the coordinate
vectors at~$p$ are precisely~$e_1$, \dots, $e_m$.
In particular, we then have~$\nabla^{TM}_{e_i}e_j=0$ for all~$i$, $j$.
Let us write~$c_i$ for Clifford multiplication with the vector~$e_i$.

The Levi-Civita connection induces a connection~$\nabla^{SM}$ on~$SM$
that is compatible with Clifford multiplication.
Let~$(V,\nabla^V)$ be as above,
and let~$R^V$ denote the curvature of~$\nabla^V$.
Note that even if~$M$ is not spin,
the curvature~$R^V\in\Omega^2(M;\End W)$ and the
twisted Riemannian Dirac operator~$D^W$ on~$\Gamma(W)$
are still well-defined.

Let~$\tau\in\Omega^3(M)$ be an alternating form with coefficients
\begin{equation}\label{AlphaFormel}
  \tau(e_i,e_j,e_k)=\tau_{ijk}\;,
\end{equation}
then the connection
\begin{equation}
  \nabla^{\prime TM}_{e_i}e_j
  =\nabla^{TM}_{e_i}e_j+\frac12\,\sum_{k=1}^m\tau_{ijk}\,e_k
\end{equation}
has alternating torsion~$T$ with~$\tau=\<T(\punkt,\punkt),\punkt\>$.

The connection~$\nabla^{\prime TM}$ induces a connection~$\nabla^{\prime SM}$
on~$SM$ such that
\begin{equation}\label{NablaXSMDef}
  \nabla^{\prime SM}_{e_i}
  =\nabla^{SM}_{e_i}+\frac18\sum_{j,k=1}^m\tau_{ijk}c_jc_k\;.
\end{equation}
We also consider another connection~$\tilde\nabla^{SM}$ on~$SM$
that is given by
\begin{equation*}
  \tilde\nabla^{SM}_{e_i}
  =\nabla^{SM}_{e_i}+\frac1{24}\sum_{j,k=1}^m\tau_{ijk}c_jc_k\;.
\end{equation*}
We then regard the modified Dirac twisted operator
\begin{equation}\label{DslFormel}
  \tilde D^W
  =\sum_{i=1}^mc_i(\tilde\nabla^{SM}\otimes\nabla^V)_{e_i}
  =\sum_{i=1}^mc_i(\nabla^{SM}\otimes\nabla^V)_{e_i}
    +\frac1{24}\sum_{i,j,k=1}^m\tau_{ijk}c_ic_jc_k\;.
\end{equation}
The square of this operator has been computed by Agricola and Friedrich.

\begin{Lemma}[\cite{AF}, Theorem~6.2]\label{BLWLemma}
  The square of~$\tilde D^W$ is given by
  \begin{multline*}
    (\tilde D^W)_p^2
    =(\nabla^{\prime SM}\otimes\nabla^V)_p^*(\nabla^{\prime SM}\otimes\nabla^V)
	+\frac\kappa4
	+\frac12\sum_{i,j=1}^mc_ic_j\,R^V_{e_i,e_j}\\
	+\frac1{96}\sum_{i,j,k,l=1}^m(d\tau)_{ijkl}\,c_ic_jc_kc_l
	-\frac1{48}\sum_{i,j,k=1}^m\tau_{ijk}^2\;.\quad\qed
  \end{multline*}
\end{Lemma}

\section{Some Curvature Formulas}\label{EstSect}
Let~$p\in M$ and let~$\bar e_1(p)$, \dots, $\bar e_m(p)$
be an orthonormal base of~$(T_pM,\bar g_p)$.
We may assume that there exist~$\lambda_1(p)$, \dots, $\lambda_m(p)>0$
such that~$\bar e_1(p)=\lambda_1(p)e_1(p)$, \dots,
$\bar e_m(p)=\lambda_m (p)e_m(p)$,
where~$e_1(p)$, \dots, $e_m(p)$ is an orthonormal base
with respect to~$g$ as before.
Because~$\bar g\ge g$ on~$\Lambda^2TM$,
we have~$\lambda_i(p)\lambda_j(p)\le 1$ for all~$p\in M$ and all~$i\ne j$.
From now on, we fix~$p$ and write simply~$\bar e_i=\lambda_ie_i$.
We may identify the spinor spaces for both metrics in such a way
that the Clifford actions at~$p$ are related by
\begin{equation*}
  c(e_i)=\bar c(\bar e_i)=:c_i\;.
\end{equation*}

If~$\nabla'$ is a metric connection with parallel
and alternating torsion~$T$ on~$(M,g)$,
we consider the alternating form
\begin{equation*}
  \tau=g\bigl(T(\punkt,\punkt),\punkt\bigr)\;.
\end{equation*}
Let us put~$\bar\tau=\tau$ for the moment---%
we will consider a different~$\bar\tau$ in Section~\ref{MapSection}.
With respect to the new metric~$\bar g$,
the coefficients of this form at~$p$ are given by
\begin{equation*}
  \bar\tau_{ijk}=\tau(\bar e_i,\bar e_j,\bar e_k)
  =\lambda_i\lambda_j\lambda_k\,\tau_{ijk}\;.
\end{equation*}

Let~$\widehat{SM}\to M$ denote a second copy of the spinor bundle
with respect to the old metric~$g$,
equipped with the connection~$\nabla^{\prime SM}$ as in~\eqref{NablaXSMDef}.
Let~$V\to M$ be a $\nabla^{\prime SM}$-parallel subbundle of~$\widehat{SM}$,
and let~$W=SM\otimes V$,
then~$W\subset\Lambda^\bullet T^*M\otimes_\R\C$ by~\eqref{LambdaFormel}.
The bundle~$W$ is a $\nabla^{\prime SM}$-parallel Dirac subbundle
of~$\Lambda^\bullet T^*M_\R\C$ as in Theorem~\ref{MainThm}.
We will write~$\hat c$ for the Clifford multiplication of~$TM$
on~$\widehat{SM}$.
Then at the point~$p$,
the curvature of~$\hat\nabla'|_V$ is given by
\begin{equation*}
  \hat R'_{\bar e_i,\bar e_j}
  =\frac14\sum_{k,l=1}^m\<R'_{\bar e_i,\bar e_j}e_k,e_l\>\,\hat c_k\hat c_l
  =\frac14\,\lambda_i\lambda_j\sum_{k,l=1}^mR'_{ijkl}\,\hat c_k\hat c_l\;.
\end{equation*}

Let~$\bar\nabla$ denote the Levi-Civita connection of the metric~$\bar g$.
We regard the connection
\begin{equation}\label{NablaxquerDef}
  \bar\nabla'_UV
  =\bar\nabla_UV+\frac12\sum_{i=1}^m\bar\tau(U,V,\bar e_i)\,\bar e_i
\end{equation}
on~$TM$.
Note that~$\bar\nabla'$ has alternating torsion~$\bar T$
with~$\bar g\bigl(\bar T(\punkt,\punkt),\punkt\bigr)=\bar\tau$,
but in general~$\bar\nabla'\bar T\ne 0$.
Let~$\bar D^W$ denote the Riemannian Dirac operator on~$(M,\bar g)$
twisted by~$(V,\nabla')$.
We modify~$\bar D^W$ as in~\eqref{DslFormel},
obtaining the operator
\begin{equation*}
  \tilde{\bar D}^W
  =\bar D^W+\frac1{24}\sum_{i,j,k=1}^m\bar\tau_{ijk}\,c_ic_jc_k
  =\bar D^W+\frac1{24}\sum_{i,j,k=1}^m\lambda_i\lambda_j\lambda_k\,
    \tau_{ijk}\,c_ic_jc_k\;.
\end{equation*}

The Bochner-Lichnerowicz-Weitzenb\"ock formula in Lemma~\ref{BLWLemma} gives
\begin{multline*}
  \bigl(\tilde{\bar D}^W\bigr)_p^2
  =(\bar\nabla'\otimes\hat\nabla')^*
      (\bar\nabla'\otimes\hat\nabla')+\frac{\bar\kappa}4
    +\frac18\sum_{i,j,k,l=1}^m\lambda_i\lambda_j\,R'_{ijkl}\,
      c_ic_j\,\hat c_k\hat c_l\\
    +\frac1{96}\sum_{i,j,k,l=1}^m
      (d\bar\tau)(\bar e_i,\bar e_j,\bar e_k,\bar e_l)\,c_ic_jc_kc_l
    -\frac1{48}\sum_{i,j,k=1}^m\bar\tau_{ijk}^2\;.
\end{multline*}
We rewrite this as
\begin{align}\label{BLWFormel}
  \begin{split}
    \bigl(\tilde{\bar D}^W\bigr)_p^2
    &=(\bar\nabla'\otimes\hat\nabla')^*
	  (\bar\nabla'\otimes\hat\nabla')+\frac{\bar\kappa}4\\
    &\qquad
	+\frac1{16}\sum_{i,j,k,l=1}^mR^{\prime TM}_{ijkl}\,
	  (\lambda_i\lambda_j\,c_ic_j+\hat c_i\hat c_j)\,
	  (\lambda_k\lambda_l\,c_kc_l+\hat c_k\hat c_l)\\
    &\qquad
	-\frac1{16}\sum_{i,j,k,l=1}^m\lambda_i\lambda_j\lambda_k\lambda_l\,
	  R^{\prime TM}_{ijkl}\,c_ic_jc_kc_l
	-\frac1{16}\sum_{i,j,k,l=1}^mR^{\prime TM}_{ijkl}\,
	  \hat c_i\hat c_j\hat c_k\hat c_l\\
    &\qquad
	+\frac1{96}\sum_{i,j,k,l=1}^m\lambda_i\lambda_j\lambda_k\lambda_l\,
	  (d\tau)_{ijkl}\,c_ic_jc_kc_l
	-\frac1{48}\sum_{i,j,k=1}^m\lambda_i^2\lambda_j^2\lambda_k^2
	  \,\tau_{ijk}^2\;.
  \end{split}
\end{align}
We will now use the various formulas in \ref{ParTorLemma}
to simplify the right hand side of~\eqref{BLWFormel}.

\begin{Proposition}\label{SquareProp}
  We have
  \allowdisplaybreaks
  \begin{align*}
    \frac1{16}\sum_{i,j,k,l=1}^m\lambda_i\lambda_j\lambda_k\lambda_l
      \,R^{\prime TM}_{ijkl}\,c_ic_jc_kc_l
    &=\frac\kappa8-\frac1{32}\sum_{i,j,k=1}^m\tau_{ijk}^2
      -\frac18\sum_{i,j=1}^m(1-\lambda_i^2\lambda_j^2)\,R'_{ijji}\\
    &\qquad
      +\frac1{96}\sum_{i,j,k,l=1}^m\lambda_i\lambda_j\lambda_k\lambda_l\,
        (d\tau)_{ijkl}\,c_ic_jc_kc_l\;,\tag1\label{AijkComparison}\\
    \frac1{16}\sum_{i,j,k,l=1}^mR^{\prime TM}_{ijkl}\,
        \hat c_i\hat c_j\hat c_k\hat c_l
    &=\frac\kappa8+\frac1{96}\sum_{i,j,k=1}^m\tau_{ijk}^2
      -\biggl(\frac1{12}\sum_{i,j,k=1}^m\tau_{ijk}
        \,\hat c_i\hat c_j\hat c_k\biggr)^2\;.\tag2\label{HatCurvLemma}
  \end{align*}
\end{Proposition}

\begin{proof}
  Using Lemma~\ref{ParTorLemma}
  and the well-known computation of the scalar curvature
  in the classical Schr\"odinger-Lichnerowicz formula,
  we compute
  \begin{multline*}
    \sum_{i,j,k,l=1}^m\lambda_i\lambda_j\lambda_k\lambda_l\,
	R^{\prime TM}_{ijkl}\,c_ic_jc_kc_l
    =2\sum_{i,j=1}^m\lambda_i^2\lambda_j^2R_{ijji}
      +\frac14\sum_{i,j,k,l=1}^m\lambda_i\lambda_j\lambda_k\lambda_l\,
	(d\tau)_{ijkl}\,c_ic_jc_kc_l\\
      +\frac14\sum_{i,j,k,l,p=1}^m\lambda_i\lambda_j\lambda_k\lambda_l\,
	(\tau_{jkp}\tau_{ipl}-\tau_{ikp}\tau_{jpl})\,c_ic_jc_kc_l\;.
  \end{multline*}
  By the Clifford relations and Lemma~\ref{ParTorLemma}~\eqref{RiccatiRel},
  this becomes
  \begin{multline*}
    \sum_{i,j,k,l=1}^m\lambda_i\lambda_j\lambda_k\lambda_l\,
	R^{\prime TM}_{ijkl}\,c_ic_jc_kc_l\\
    \begin{aligned}
      &=2\sum_{i,j=1}^m\lambda_i^2\lambda_j^2R_{ijji}
	+\frac14\sum_{i,j,k,l=1}^m\lambda_i\lambda_j\lambda_k\lambda_l\,
	  (d\tau)_{ijkl}\,c_ic_jc_kc_l
	-\frac12\sum_{i,j,k=1}^m\,\lambda_i^2\lambda_j^2\,\tau_{ijk}^2\\
      &\qquad
	-\frac16\sum_{i,j,k,l,p=1}^m\lambda_i\lambda_j\lambda_k\lambda_l\,
	  (\tau_{ijp}\tau_{pkl}+\tau_{ikp}\tau_{plj}
	+\tau_{ilp}\tau_{pjk})\,c_ic_jc_kc_l\\
      &=2\sum_{i,j=1}^m\lambda_i^2\lambda_j^2R_{ijji}
	-\frac12\sum_{i,j,k=1}^m\,\lambda_i^2\lambda_j^2\,\tau_{ijk}^2
	+\frac16\sum_{i,j,k,l=1}^m\lambda_i\lambda_j\lambda_k\lambda_l\,
	  (d\tau)_{ijkl}\,c_ic_jc_kc_l\;.
    \end{aligned}
  \end{multline*}
  Using Lemma~\eqref{ParTorLemma}~\ref{SectRel},
  we obtain~\eqref{AijkComparison}.

  To prove~\eqref{HatCurvLemma},
  we may take formulas established above
  and replace all Clifford variables~$c_i$ by~$\hat c_i$ and vice versa.
  Either by direct computation or by collecting those terms
  in Lemma~\ref{BLWLemma} where~$\tau$ occurs quadratically,
  we obtain
  \begin{equation*}
    \biggl(\frac1{24}\sum_{i,j,k=1}^m\tau_{ijk}\,\hat c_i\hat c_j\hat c_k
	\biggr)^2
    =-\sum_{i=1}^m
	\biggl(\frac18\sum_{j,k=1}^m\tau_{ijk}\hat c_j\hat c_k\biggr)^2
      -\frac1{48}\sum_{i,j,k=1}^m\tau_{ijk}^2\;.
  \end{equation*}
  Using equation~\eqref{CurvatureFormel2} for~$R'$ and proceeding as above,
  we find
  \begin{align*}
    \sum_{i,j,k,l=1}^mR'_{ijkl}\,\hat c_i\hat c_j\hat c_k\hat c_l
    &=2\kappa-\frac14\sum_{i,j,k,l,p=1}^m
	(2\tau_{ijp}\tau_{kpl}+\tau_{jkp}\tau_{ipl}+\tau_{kip}\tau_{jpl})
	\,\hat c_i\hat c_j\hat c_k\hat c_l\\
    &=2\kappa
      +\sum_{i=1}^m\biggl(\sum_{j,k=1}^m\tau_{ijk}\,\hat c_i\hat c_j\biggr)^2
      +\frac32\sum_{i,j,k=1}^m\tau_{ijk}^2\\
    &=2\kappa
      -\biggl(\frac13\sum_{i,j,k=1}^m\tau_{ijk}
	\,\hat c_i\hat c_j\hat c_k\biggr)^2
      +\frac16\sum_{i,j,k=1}^m\tau_{ijk}^2\;,
  \end{align*}
  which proves~\eqref{HatCurvLemma}.
\end{proof}

\section{A Bochner-Lichnerowicz-Weitzenb\"ock Formula}\label{BLWSection}

We now establish Corollary~\ref{HolonomyCor}.
This chapter is not needed in the proof of Theorem~\ref{MainThm}.
We regard the special case~$\bar g=g$.
Then the operator~$\tilde{\bar D}^W$ becomes a modified Euler-
or signature operator~$\tilde D^W$.
In this case,
we can use~\eqref{BLWFormel} and Proposition~\ref{SquareProp}
to simplify the Bochner-Lichnerowicz-Weitzenb\"ock formula
in Lemma~\ref{BLWLemma} as follows.

\begin{Corollary}\label{BLWCor}
  Assume that~$\nabla'$ has parallel and alternating torsion.
  Then the square of the modified signature operator~$\tilde D^W$ is given by
  \begin{multline*}
    \bigl(\tilde D^W\bigr)_p^2
    =(\nabla'\otimes\hat\nabla')^*
	  (\nabla'\otimes\hat\nabla')
	+\biggl(\frac1{12}\sum_{i,j,k=1}^m\tau_{ijk}
          \,\hat c_i\hat c_j\hat c_k\biggr)^2\\
	+\frac1{16}\sum_{i,j,k,l=1}^mR^{\prime TM}_{ijkl}\,
	  (c_ic_j+\hat c_i\hat c_j)\,
	  (c_kc_l+\hat c_k\hat c_l)\;.\quad\qed
  \end{multline*}
\end{Corollary}

If~$R'$ acts nonnegatively on~$\Lambda^2TM$,
then this formula implies that all $\tilde D^W$-harmonic forms
are $\nabla'$-parallel.
We will now prove Corollary~\ref{HolonomyCor},
which allows us to compute some global topological invariants of~$M$
at a single point.


\begin{proof}[Proof of Corollary~\ref{HolonomyCor}]
  We choose~$(V,\nabla^V)=(SM,\nabla^{\prime SM})$.
  The operator
  \begin{equation*}
    \sum_{i,j,k=1}^m\tau_{ijk}\,\hat c_i\hat c_j\hat c_k
  \end{equation*}
  is selfadjoint,
  and hence has a positive square.

  If~$R'$ acts nonnegatively on~$\Lambda^2TM$,
  it possesses a symmetric square root that we will denote~$B$,
  with coefficients~$B_{ijkl}$.
  More precisely,
  \begin{equation*}
    -R_{ijkl}=\sum_{p,q=1}^mB_{ijpq}\,B_{pqkl}\;.
  \end{equation*}
  Thus,
  \begin{equation}\label{RootFormel}
    \sum_{i,j,k,l=1}^mR^{\prime TM}_{ijkl}\,
      (c_ic_j+\hat c_i\hat c_j)\,
      (c_kc_l+\hat c_k\hat c_l)
    =-\sum_{i,j=1}^m\biggl(\sum_{k,l=1}^MB_{ijkl}\,
      (c_kc_l+\hat c_k\hat c_l)\biggr)^2
    \ge 0
  \end{equation}
  because the square of a skewadjoint operator is negative.

  As all operators on the right hand side in the Corollary
  are nonnegative,
  each $\tilde D^W$-harmonic form lies in the kernel of
  the connection Laplacian, and is thus $\nabla'$-parallel.
  Note that
  \begin{equation*}
    \tilde D=\sum_{i=1}^mc_i(\nabla'\otimes\hat\nabla')_{e_i}
      -\frac1{12}\sum_{i,j,k=1}^m\tau_{ijk}\,c_ic_jc_k\;.
  \end{equation*}
  Clearly,
  the first term on right hand side vanishes on $\nabla'$-parallel forms.
  Because the torsion is $\nabla'$-parallel and $\nabla'$ is
  compatible with Clifford multiplication,
  we see that the second term on the right hand side is $\nabla'$-parallel.
  In particular,
  the operator~$\tilde D$ preserves the space~$\ker\nabla'$
  of $\nabla'$-parallel forms.

  Assuming that~$m$ is even,
  we conclude
  \begin{multline*}
    \chi(M)=\ind(\tilde D)=\ind(\tilde D|_{\ker\nabla'})
    =\dim\bigl(\ker\nabla'\cap\Omega^{\even}(M)\bigr)\\
      -\dim\bigl(\ker\nabla'\cap\Omega^{\odd}(M)\bigr)
    =\dim(\Lambda^\even\pi)^H-\dim(\Lambda^\odd\pi)^H\;,
  \end{multline*}
  because $\nabla'$-parallel forms are determined by their value
  at a single point.
  The signature and the Kervaire semicharacteristic can be computed
  similarly if they are defined.
\end{proof}

\section{Normal Homogeneous Spaces}\label{HomoSect}

We regard homogeneous spaces as a special case of Riemannian manifolds
that admit a connection with nonvanishing parallel and alternating torsion.

Let~$G$ be a Lie group, let~$H\subset G$ be a closed subgroup,
and let~$\frh\subset\frg$ denote their Lie algebras.
In the following,
we consider the homogeneous space~$M=G/H$,
with tangent bundle~$TM\cong G\times_H(\frg/\frh)$.
Let~$g$ be a metric on~$M$,
then~$g$ is invariant under the natural action of~$G$
iff~$g$ is induced by an $\Ad_H$-invariant metric~$g_{\frg/\frh}$
on~$\frg/\frh$.
In this case,
we call~$(M,g)$ a Riemannian homogeneous space.

\begin{Definition}
  A Riemannian homogeneous space~$(G/H,g)$ is called
  {\em naturally reductive\/}
  if there exists an $\Ad_H$-invariant linear complement~$\frp$ of~$\frh$
  in~$\frg$ with an $\Ad_H$-invariant metric~$g_\frp$ inducing
  the Riemannian metric~$g$ on~$T(G/H)\cong G\times_H\frp$,
  such that the three-form~$g_\frp([\punkt,\punkt]_\frp,\punkt)$ on~$\frp$
  is fully alternating.

  A naturally reductive Riemannian homogeneous space~$(G/H,g)$
  is called {\em normal\/} if there exists an $\Ad_G$-invariant
  metric~$g_\frg$ on~$\frg$ such that~$\frp=\frh^\perp$
  and~$g_\frp=g_\frg|_\frp$.
 \end{Definition}

Let~$(M,g)=(G/H,g)$ be naturally reductive,
then we have the isotropy representation
\begin{equation*}
  \pi=\Ad|_{H\times\frp}\colon H\to\SO(\frp)\;.
\end{equation*}
We identify~$\frg$ with the left invariant vector fields on~$G$,
then the action of~$H$ on~$G$ is generated by the left invariant vector fields
corresponding to elements of~$\frh\subset\frg$.
We represent a vector field~$X$ on~$M$ by its horizontal lift~$\hat X$ to~$G$.
Then~$\hat X$ is a map~$\hat X\colon G\to\frp\subset\frg$
that is $H$-equivariant with respect to the isotropy representation.
Then the {\em reductive connection\/}~$\nabla'$ and
the Levi-Civita connection~$\nabla^{TM}$ are given by
\begin{equation*}
  \widehat{\nabla^{\prime TM}_XY}=\hat X(\hat Y)
  \qquad\text{and}\qquad
  \widehat{\nabla^{TM}_XY}=\hat X(\hat Y)+\frac12\,[\hat X,\hat Y]_\frp\;,
\end{equation*}
where~$[\hat X,\hat Y]_\frp$ is the pointwise algebraic Lie bracket
projected to~$\frp$,
see the Cheeger and Ebin~\cite{CE}.
These connections are different unless~$(M,g)$ is locally symmetric.
Let us recall some well-known facts about normal homogeneous metrics.

\begin{Lemma}\label{RedLemma}
  Let~$M=G/H$ be a homogeneous space with a normal metric~$g$.
  Then the reductive connection~$\nabla'$ has parallel and alternating
  torsion,
  and its curvature operator~$R'$ on~$\Lambda^2TM$ is nonnegative.
\end{Lemma}

\begin{proof}
  One easily checks that the torsion tensor~$T$ of~$\nabla^{\prime TM}$
  satisfies
  \begin{equation*}
    \widehat{T(X,Y)}=-[\hat X,\hat Y]_\frp\;.
  \end{equation*}
  In particular, $\nabla^{\prime TM}$ has alternating torsion
  by natural reductivity.
  The torsion is parallel because
  \begin{equation*}
    \hat X([\hat Y,\hat Z]_\frp)
    =[\hat X(\hat Y),\hat Z]_\frp+[\hat Y,\hat X(\hat Z)]_\frp\;.
  \end{equation*}

  The reductive connection has curvature
  \begin{equation*}
    \widehat{R'_{X,Y}Z}=-\bigl[[\hat X,\hat Y]_\frh,\hat Z\bigr]_\frp\;.
  \end{equation*}
  On a normal homogeneous space,
  we may write
  \begin{equation}\label{HomoCurvature}
    \<R'_{X,Y}Z,W\>=-\<[\hat X,\hat Y]_\frh,[\hat Z,\hat W]_\frh\>\;,
  \end{equation}
  in particular,
  the action of~$R'$ on~$\Lambda^2TM$ is non-negative.
  This is not true in general for naturally reductive spaces.  
\end{proof}

For the rigidity statements in Theorems~\ref{MainThm} and~\ref{MapThm},
we need to control the Ricci curvature.

\begin{Lemma}\label{RicciLemma}
  Let~$M=G/H$ and~$g$ be as before.
  Then the Ricci curvature~$\rho$ is positive definite unless~$G/H$
  contains a Euclidean local de Rham factor,
  which implies in particular that~$\rk G>\rk H$.
  Moreover, the form~$2\rho-\kappa\,g$ is negative definite
  unless~$M$ is covered by~$\R^m$ or~$S^2\times\R^{m-2}$.
\end{Lemma}

\begin{proof}
  The Riemannian curvature tensor of~$(M,g)$ satisfies
  \begin{equation*}
    \<R_{X,Y}Y,X\>
    =\bigl\|[\hat X,\hat Y]_\frh\bigr\|^2
    +\frac14\,\bigl\|[\hat X,\hat Y]_\frp\bigr\|^2
    \ge 0\;.
  \end{equation*}
  This implies that the Ricci tensor is nonnegative.

  Assume that~$X(p)\in\ker\rho$ is a unit vector,
  and extend~$\hat X(p)$ to a basis~$e_1=\hat X(p)$, $e_2$, \dots, $e_m$
  of~$\frp$.
  Because
  \begin{equation*}
    0=\rho(X,X)(p)
    =\sum_{i=2}^m\biggl(\norm{[e_1,e_i]_\frh}^2
      +\frac14\,\norm{[e_1,e_i]_\frp}^2\biggr)
  \end{equation*}
  and $\<R_{X,Y}Y,X\>=0$ implies~$[\hat X,\hat Y]_\frg=0$,
  we conclude that~$[e_1,w]_\frg=0$ for all~$w\in\frp$.
  For~$v\in\frh$,
  we have~$[e_1,v]\in\frp$ and
  \begin{equation*}
    \<[e_1,v],w\>=-\<[e_1,w],0\>=0
  \end{equation*}
  for all~$w\in\frp$.
  Hence~$e_1$ lies in the centrum~$\frz$ of~$\frg$.
  But then the vector field~$X_0$ with~$\hat X_0\equiv e_1$
  is a parallel vector field that spans a one-dimensional
  Euclidean local de Rham factor of~$M$.
  If~$G$ and~$H$ are of equal rank,
  then~$\frz\subset\frh$,
  so~$\frz\cap\frp\ne0$ implies~$\rk G>\rk H$.

  By Remark~\ref{RigidityRem},
  we conclude that~$2\rho-\kappa\,g$ is negative definite
  unless~$T=0$.
  Let us assume~$T=0$,
  then~$M=G/H$ is symmetric,
  and the universal covering~$\tilde M$ of~$M$ splits as a product
  of irreducible symmetric spaces of compact type and~$\R^k$ for some~$k\ge0$.
  Let~$\rho_1$, \dots, $\rho_m$ denote the eigenvalues of the Ricci tensor.
  As explained in~\cite{GS},
  the form~$(2\rho-\kappa)$ is negative definite unless
  \begin{equation*}
    \rho_i\ge\sum_{j\ne i}\rho_j
  \end{equation*}
  for some index~$i\in\{1,\dots,m\}$.
  But since every irreducible symmetric space is Einstein,
  this is only possible if~$\tilde M=\R^m$ or~$\tilde M=S^2\times\R^{m-2}$.
\end{proof}

\begin{proof}[Proof of Corollary~\ref{HomoThm}]
  It remains to control the index of the Hodge-Dirac operator
  on suitable Dirac subbundles of~$\Lambda^\bullet T^*M$
  if~$\rk G=\rk H$.
  One has the well-known formula
  \begin{equation}
    \chi(G/H)=\#(W_G/W_H)>0\;,
  \end{equation}
  where~$W_G$ and~$W_H$ are the Weyl groups of~$G$ and~$H$
  with respect to a common maximal torus.
  In particular,
  the Euler operator itself has nonvanishing index,
  see Example~\ref{IndexExample}~\eqref{EulerCond}.
\end{proof}

If~$\rk G-\rk H=1$ and~$\dim M\equiv 1$ mod~$4$,
the situation is more complicated.
There are example like~$S^{4k+1}$ where the Kervaire semicharacteristic
does not vanish.
In Example~\ref{OddIndexExample},
we have seen that for product spaces like~$M=G/H\times S^{4k+2}$,
there is no suitable $\nabla'$-parallel Dirac
subbundle~$W\subset\Lambda^\bullet T^*M$ giving rise to a nontrivial index.

Corollary~\ref{HolonomyCor} suggests to treat the case~$\rk G-\rk H=1$
and~$\dim M=4k+1$ by representation theoretic methods.
First note that if~$H$ is connected and~$\frh$ contains no central
elements of~$\frg$,
then~$H$ is the global holonomy group of the reductive connection on~$TM$,
in accordance with our notation in Corollary~\ref{HolonomyCor}.

We regard the infinitesimal isotropy representation of~$M$
as a Lie algebra homomorphism~$\pi_*\colon\frh\to\spin(m)$.
Pulling back the complex spinor representation of~$\spin(m)$ to~$\frh$
gives rise to the spinor representation~$\tilde\pi_*$ of~$M$.
Let~$\hat{\tilde\pi}_*$ be another copy of~$\tilde\pi_*$.
By~\eqref{LambdaFormel},
the complex exterior representation~$\Lambda^\bullet\pi_*\otimes_\R\C$
is isomorphic
to the sum of one or two copies of~$\tilde\pi_*\otimes\hat{\tilde\pi}_*$.
If~$M$ is $G$-equivariantly spin,
this integrates to an isomorphism of $H$-representations.

Let~$S\subset T$ be maximal tori of~$H$ and~$G$, respectively.
Let~$\frs\subset\frt$ denote their Lie algebras.
Because~$g$ is a normal metric,
we may identify the dual~$\frs^*$ with a subspace of~$\frt^*$.
Let~$\rho_H\in\frs^*$ and~$\rho_G\in\frt^*$ denote the half sums
of the positive (real) roots of~$H$ and~$G$ with respect to some Weyl chamber.
By abuse of notation,
we identify each irreducible representation of~$H$ or~$G$
with its heighest (real) weight in~$\frs^*$ or~$\frt^*$, respectively.

\begin{Remark}\label{HomoBLWExp}
  Let us regard the Bochner-Lichnerowicz-Weitzenb\"ock formula
  in Corollary~\ref{BLWCor}.
  The group~$G$ acts on~$\Lambda^\bullet T^*M$,
  and the restriction of the infinitesimal action to~$\frh$ is isomorphic
  to~$\tilde\pi_*\otimes\hat{\tilde\pi}_*$.
  Using~\eqref{HomoCurvature}, we can identify the term
  \begin{equation*}
    (\nabla'\otimes\hat\nabla')^*
	  (\nabla'\otimes\hat\nabla')
	+\frac1{16}\sum_{i,j,k,l=1}^mR^{\prime TM}_{ijkl}\,
	  (c_ic_j+\hat c_i\hat c_j)\,
	  (c_kc_l+\hat c_k\hat c_l)
  \end{equation*}
  with the action of the Casimir operator~$c_G$.
  By a computation in~\cite{G1}, \cite{G2},
  we have
  \begin{equation*}
    \biggl(\frac1{12}\sum_{i,j,k=1}^m\tau_{ijk}
      \,\hat c_i\hat c_j\hat c_k\biggr)^2
    =\norm{\rho_G}^2-\norm{\rho_H}^2-c_H^{\hat{\tilde\pi}}\;,
  \end{equation*}
  where~$c_H^{\hat{\tilde\pi}}$ denotes the Casimir operator of~$H$.
  If~$\kappa$ is an irreducible component of~$\hat{\tilde\pi}_*$
  acting on~$V^\kappa$,
  let~$V=G\times_HV^\kappa\subset SM$,
  and consider the Dirac subbundle~$W\subset\Lambda^\bullet T^*M\otimes_\R\C$.
  If~$\gamma$ is an irreducible representation of~$G$,
  then~$(\tilde D^W)^2$ acts on the $\gamma$-isotypical component
  of~$\Gamma(W)$ by
  \begin{equation*}
    c_G+\norm{\rho_G}^2-\norm{\rho_H}^2-c_H^\kappa
    =\norm{\gamma+\rho_G}^2-\norm{\kappa+\rho_H}^2\;.
  \end{equation*}
  This is the well-known generalisation of Parthasarathy's formula~\cite{Parth}
  to homogeneous spaces.
\end{Remark}

\begin{Lemma}\label{HomoIndLemma}
  Let~$M=G/H$ be a quotient of compact Lie groups.
  Then~$\Lambda^\bullet TM$ contains a Dirac subbundle~$W$
  with~$\ker\tilde D^W\ne 0$
  if and only if there exists~$w\in W_G$ such that~$w(\rho_G)\in\frs^*$.
  The index of~$\tilde D^W$ vanishes if~$\rk G-\rk H>1$.
\end{Lemma}

\begin{proof}
  By Lemma~4.4 in~\cite{G1}, for each irreducible subrepresentation~$\kappa$
  of~$\tilde\pi$ on~$S$, we have~$\norm{\kappa+\rho_H}\le\norm{\rho_G}$.
  Let~$W$ be the associated Dirac subbundle of~$\Lambda^\bullet T^*M$
  as in Remark~\ref{HomoBLWExp}.
  Then on the $\gamma$-isotypical component,
  the operator~$(\tilde D^W)^2$ acts by the scalar
  \begin{equation}\label{KappaNormFormel}
    \norm{\gamma+\rho_G}^2-\norm{\kappa+\rho_H}^2
    \ge\norm{\gamma+\rho_G}-\norm{\rho_G}^2
    =c_G^\gamma\ge 0\;.
  \end{equation}
  Equality can only occur if~$\gamma$ is the trivial representation
  and if~$\norm{\kappa+\rho_H}=\norm{\rho_G}$.
  By~\cite{G1},
  the equation~$\norm{\kappa+\rho_H}=\norm{\rho_G}$ holds for
  a subrepresentation~$\kappa$ of~$\tilde\pi$
  iff~$w\rho_G\in\frs^*$ and~$\kappa=w\rho_G-\rho_H$
  for some element~$w$ in the Weyl group of~$G$.
  This is always the case if~$\frs=\frt$,
  i.e., if~$\rk G=\rk H$.

  Now assume that there exists~$\kappa\subset\tilde\pi$
  such that~$\ker(\tilde D^\kappa)\ne0$.
  By~\eqref{KappaNormFormel},
  one easily concludes that~$\dim_\C\ker(\kappa)$ equals the multiplicity
  of~$\kappa$ in~$\tilde\pi$.
  By Lemma~4.4 in~\cite{G1},
  this multiplicity is~$1$ if~$\rk G\le\rk H+1$
  and no nonzero weight of~$\ad_\frg$ vanishes on~$\frs$.

  On the other hand,
  if~$\rk G\ge\rk H+2$,
  then the multiplicity is even,
  and if~$\dim M$ is even,
  the multiplicity of~$\kappa$ in~$\tilde\pi^+$ and~$\tilde\pi^-$ is the same.
  In this case,
  the index of~$\tilde D^\kappa$ is zero.
\end{proof}

Unfortunately, Lemma~\ref{HomoIndLemma} only states a necessary condition.
There are homogeneous spaces~$G/H$ where~$w\rho_G\notin\frs$
for all~$w\in W_G$.
The Berger space~$\SO(5)/\SO(3)$ considered in~\cite{G1}, \cite{G2}
is an example.

Assume that~$\rk G-\rk H=1$, $\dim M=4k+1$, $w\rho_G\in\frs^*$
and the multiplicity of~$\kappa=w\rho_G-\rho_H$ in~$\tilde\pi_*$
is one.
The example~$M=S^2\times S^3$ shows
that~$\tilde D^W$ can still have vanishing index.
A careful investigation of the bundles of real differential forms
and real spinors using the classification in~\cite{LM} shows
that~$\kappa$ must be of quaternionic type if~$k$ is odd,
and of real type if~$k$ is even.
In the cases of Example~\ref{OddIndexExample},
the corresponding representations are of complex type.

\section{Proof of the Scalar Curvature Estimate}\label{ProofSect}
We prove Theorem~\ref{MainThm} for even-dimensional~$M$.

As in Corollary~\ref{BLWCor} above,
we insert the formulas of Proposition~\ref{SquareProp}
into~\eqref{BLWFormel}.
Using~\eqref{RootFormel}, we obtain
\begin{align}\label{DquerBLWFormel}
  \begin{split}
    \bigl(\tilde{\bar D}^W\bigr)_p^2
    &=(\bar\nabla'\otimes\hat\nabla')^*
	(\bar\nabla'\otimes\hat\nabla')+\frac{\bar\kappa-\kappa}4
      +\biggl(\frac1{12}\sum_{i,j,k=1}^m\tau_{ijk}
	\,\hat c_i\hat c_j\hat c_k\biggr)^2\\
    &\qquad
      -\frac1{16}\sum_{i,j=1}^m\biggl(\sum_{k,l=1}^mB_{ijkl}\,
	(\lambda_k\lambda_l\,c_kc_l+\hat c_k\hat c_l)\biggr)^2\\
    &\qquad
      +\frac18\sum_{i,j=1}^m(1-\lambda_i^2\lambda_j^2)\,R'_{ijji}
      +\frac1{48}\sum_{i,j,k=1}^m(1-\lambda_i^2\lambda_j^2\lambda_k^2)
	\,\tau_{ijk}^2\;.
  \end{split}
\end{align}
Because~$\bar g\ge g$ on~$\Lambda^2TM$ by assumption,
we have
\begin{equation}\label{LambdaProduct}
  \lambda_i\lambda_j\le 1
  \qquad\text{for all~$i$, $j\in\{1,\dots,m\}$ with~$i\ne j$.}
\end{equation}
Because we also assumed~$R'_{ijji}\ge 0$,
the last two terms above are nonnegative.
Thus finally,
\begin{equation}\label{BLWEstimate}
  \bigl(\tilde{\bar D}^W\bigr)_p^2
  \ge\frac{\bar\kappa}4-\frac\kappa4\;.
\end{equation}

By our assumptions in Theorem~\ref{MainThm}
the operator~$\tilde{\bar D}^W$ has a nontrivial kernel.
We may thus apply~\eqref{DquerBLWFormel} to a $\bar D^\kappa$-harmonic
spinor~$\sigma\ne 0$,
obtaining
\begin{align*}
  0
  \ge\int_M\biggl(\frac{\bar\kappa}4-\frac\kappa4\biggr)
    \,\norm{\sigma}^2\,d\vol_{\bar g}\;.
\end{align*}
Note that a nontrivial harmonic spinor on a connected manifold
is nonzero almost everywhere by~\cite{Baer}.
Hence if we assume that~$\bar\kappa\ge\kappa$ everywhere,
we immediately get~$\bar\kappa=\kappa$ almost everywhere,
and hence everywhere by continuity.
This completes the proof of the first part of Theorem~\ref{MainThm}.

We will now prove the second part of Theorem~\ref{MainThm}.
Note that since the torsion~$T$ is assumed to be parallel and~$M$ is connected,
either~$T=0$ or~$T$ vanishes nowhere.
In the case~$T=0$,
the hypotheses in Theorem~\ref{MainThm} are the same as in~\cite{GS}.
Hence in this case, nothing is left to prove.

Thus, we may assume that~$T\ne 0$ everywhere on~$M$.
At a point~$p\in M$,
choose an adapted $g$-orthonormal frame~$e_1$, \dots, $e_m$
as in Section~\ref{EstSect}.
Then there exist indices~$1\le p$, $q$, $r\le m$
such that~$\tau_{pqr}\ne 0$.
We may assume that~$p=1$, $q=2$, $r=3$.

Now we assume that~$\bar\kappa\ge\kappa$.
This implies that we have equality in~\eqref{BLWEstimate}
In particular, the nonnegative term
\begin{equation*}
  \frac1{48}\sum_{i,j,k=1}^m\bigl(1-\lambda_i^2\lambda_j^2\lambda_k^2\bigr)
    \,\tau_{ijk}^2
\end{equation*}
in~\eqref{DquerBLWFormel} vanishes.
For the indices~$1$, $2$, $3$, this implies
\begin{equation*}
  1=\lambda_1^2\lambda_2^2\lambda_3^2
   =(\lambda_1\lambda_2)\,(\lambda_1\lambda_3)\,(\lambda_2\lambda_3)
\end{equation*}
because~$\tau_{1,2,3}\ne 0$.
Hence~$\lambda_1=\lambda_2=\lambda_3=1$ by~\eqref{LambdaProduct}.
For the same reason, we have~$\lambda_i=1$ whenever there exist~$j$, $k$
such that~$\tau_{ijk}\ne 0$.

The assumption~$\lambda_1\lambda_i\le 1$ implies
\begin{equation}\label{LambdaComparison}
  \lambda_i\le 1
\end{equation}
for all indices~$i>1$, hence for all~$i$.
  
Now, consider an index~$i$ such that~$e_i\in\ker\tau$,
so~$\tau_{ijk}=0$ for all~$j$, $k$.
Because~$\rho|_{\ker\tau}$ is positive definite,
Lemma~\ref{ParTorLemma}~\eqref{SectRel} implies
\begin{equation*}
  0<\sum_{j=1}^mR_{ijji}
  =\sum_{j=1}^m\biggl(R_{ijji}+\sum_{k=1}^m\alpha_{ijk}^2\biggr)
  =\sum_{j=1}^mR'_{ijji}\;.
\end{equation*}
By Lemma~\ref{ParTorLemma}~\eqref{SectRel},
all~$R_{ijji}\ge0$, as noted in Remark~\ref{PosRem}.
By the above, there exists~$j$ such that~$R_{ijji}>0$.
But by~\eqref{DquerBLWFormel},
equality in~\eqref{BLWEstimate} implies
\begin{equation*}
  0=(1-\lambda_i^2\lambda_j^2)\,R_{ijji}\;,
\end{equation*}
hence~$\lambda_i\lambda_j=1$.
By~\eqref{LambdaComparison},
we have~$\lambda_i=\lambda_j=1$.
This completes the proof of the rigidity part of Theorem~\ref{MainThm}.

\section{Area Nonincreasing Spin Maps of Nonzero
\texorpdfstring{$\hat A$}{A}-Degree}
\label{MapSection}

We now explain and prove Theorem~\ref{MapThm}.
Recall that a map~$f\colon (N,\bar g)\to(M,g)$ between Riemannian manifolds
is called {\em area-nonincreasing\/} iff~$\bar g\ge f^*g$ on~$\Lambda^2TN$.
Riemannian submersion are a special case.

A {\em spin map\/} is map~$f\colon N\to M$ between differentiable manifolds
such that their second Stiefel-Whitney classes are related~$w_2(N)=f^*w_2(M)$.
This is the case precisely if there exists a Dirac bundle~$W\to N$
that is locally isometric to~$SN\otimes f^*SM$.

Finally,
if~$M$ and~$N$ are oriented with orientation classes~$(M)\in H^m(M;\Z)$
and~$[N]\in H_n(M;\Z)$,
the $\hat A$-degree of a map~$f\colon N\to M$ is defined as
\begin{equation*}
  \deg_{\hat A}f=\bigl(\hat A(N)\wedge(M)\bigr)[N]\;.
\end{equation*}
If~$f$ is a spin map and~$\bar D$ is the Dirac operator
on the bundle~$W\to N$ above,
then
\begin{equation}\label{MapIndexFormel}
  \ind(\bar D)=\deg_{\hat A}f\cdot\chi(M)\;.
\end{equation}
In Theorem~\ref{MapThm},
we have assumed that~$f\colon(N,\bar g)\to(M,g)$ is an area-nonincreasing
spin map of nonzero $\hat A$-genus,
and that~$\chi(M)\ne 0$.

\begin{proof}[Proof of Theorem~\ref{MapThm}]
  We modify the Dirac operator~$\bar D$ on~$W\to N$
  as in Section~\ref{EstSect},
  now using the three-form~$\bar\tau=f^*\tau\in\Omega^3(N)$.
  At a point~$q\in N$,
  we may fix orthonormal bases~$e_1$, \dots, $e_m$ of~$T_{f(q)}M$
  and~$\bar e_1$, \dots, $\bar e_n$ of~$T_qN$ such that there exists
  numbers~$\lambda_1$, \dots, $\lambda_m\ge 0$ with
  \begin{equation}\label{EiquerDef}
    df_q\bar e_i=
    \begin{cases}
      \lambda_ie_i	&\text{if~$1\le i\le m$, and}\\
      0			&\text{if~$i>m$.}
    \end{cases}
  \end{equation}
  Then the coefficients of~$\bar\tau$ are again given by
  \begin{equation*}
    \bar\tau_{ijk}=
    \begin{cases}
      \lambda_i\lambda_j\lambda_k\,\tau_{ijk}
	&\text{if~$1\le i$, $j$, $k\le m$, and}\\
      0	&\text{otherwise.}
    \end{cases}
  \end{equation*}
  
  By~\eqref{MapIndexFormel} and our assumptions~$\deg_{\hat A}f\ne 0$
  and~$\chi(M)\ne0$, the modified operator~$\tilde{\bar D}^W$
  has nonzero index
  \begin{equation*}
    \ind\bigl(\tilde{\bar D}^W\bigr)
    =\bigl(\hat A(TN)\wedge\ch(W/S)\bigr)[N]
    =\bigl(\hat A(TN)\wedge f^*e(TM)\bigr)[N]
    =\deg_{\hat A}(f)\cdot\chi(M)\ne 0\;.
  \end{equation*}
  From this point on,
  the proof of area-extremality in the sense of Gromov
  proceeds as in Section~\ref{ProofSect}.

  To prove strong area-extremality in the sense of Gromov,
  we again consider the modified operator~$\tilde{\bar D}$ on~$W\to N$.
  As in the proof of strong area-extremality above,
  we conclude that~$\bar\kappa\ge\kappa\circ f$
  implies that~$\lambda_1=\dots=\lambda_m=1$,
  hence~$f$ is a Riemannian submersion.
\end{proof}

\bibliographystyle{alpha}

\end{document}